\documentclass[11pt]{amsart}

\usepackage[a4paper,margin=1in]{geometry}
\usepackage[utf8]{inputenc}
\usepackage[T1]{fontenc}
\usepackage{lmodern}
\usepackage{microtype}
\usepackage{amsmath,amssymb,amsthm}
\usepackage[colorlinks=true,linkcolor=blue,citecolor=blue,urlcolor=blue]{hyperref}

\newtheorem{theorem}{Theorem}[section]
\newtheorem{proposition}[theorem]{Proposition}
\newtheorem{corollary}[theorem]{Corollary}
\newtheorem{lemma}[theorem]{Lemma}
\theoremstyle{definition}
\newtheorem{definition}[theorem]{Definition}
\newtheorem{example}[theorem]{Example}
\theoremstyle{remark}
\newtheorem{remark}[theorem]{Remark}

\newcommand{\Spec}{\operatorname{Spec}}
\newcommand{\Reg}{\operatorname{Reg}}
\newcommand{\Jac}{\operatorname{Jac}}
\newcommand{\Max}{\operatorname{Max}}

\title[Strongly multiplicative sets]{Strongly multiplicative sets, idempotent localizations, and $S$-prime phenomena}

\author{Hwankoo Kim}
\address{Division of Computer Engineering, Hoseo University, Asan, Republic of Korea}
\email{hkkim@hoseo.edu}

\author{Suat Ko\c{c}}
\address{Department of Mathematics, Marmara University, Istanbul, Turkiye}
\email{suat.koc@marmara.edu.tr}


\begin{document}

\begin{abstract}
A multiplicative set $S$ of a commutative ring $R$ is strongly multiplicative if every family $(s_i)_{i \in I}$ of elements of $S$ admits a common multiple in $S \cap \bigcap_{i \in I} s_iR$.
We combine the structural results on strongly multiplicative sets with their prime-theoretic and module-theoretic applications.
We prove that $S$ is strongly multiplicative if and only if localization at $S$ commutes with arbitrary intersections of ideals, if and only if $R_S$ is the localization at an idempotent, and if and only if $D(S)$ is clopen in $\Spec(R)$.
We then develop permanence results under homomorphisms, products, factor rings, trivial extensions, and amalgamations; describe the associated split torsion theory; show that almost multiplicative sets contribute no new cases beyond their multiplicative hull; and record a Mittag--Leffler refinement for intersections of submodules in finitely generated modules.
On the prime-theoretic side, we relate strongly multiplicative sets to strongly prime ideals and strongly zero-dimensional rings, answer the Hamed--Malek question on the role of strong multiplicativity for chains of $S$-prime ideals, prove a strong Krull separation lemma together with a maximal-ideal correspondence for $R_S$, and connect the theory with the regular $m$-complement operator.
In particular, every nontrivial strongly multiplicative localization must invert a zero divisor.
\end{abstract}

\maketitle

\section{Introduction}

Let $R$ be a commutative ring with $1 \neq 0$, and let $S \subseteq R$ be a multiplicative set, so $0 \notin S$, $1 \in S$, and $ss' \in S$ for all $s,s' \in S$.
Localization at $S$ is one of the basic constructions of commutative algebra, but it does not preserve every infinite ideal-theoretic operation.
For example, if $R=\mathbb{Z}$, $S=\mathbb{Z}\setminus\{0\}$, and $I_n=n!\mathbb{Z}$, then
\[
\mathbb{Q}\Bigl(\bigcap_{n \geq 1} I_n\Bigr)=0
\qquad\text{while}\qquad
\bigcap_{n \geq 1}\mathbb{Q}I_n=\mathbb{Q}.
\]
Thus localization need not commute with arbitrary intersections of ideals.

In \cite{Hamed}, the notion of a strongly multiplicative set was introduced and applied to the study of chains of $S$-prime ideals and minimal $S$-prime ideals.
The main structural point of the present paper is that strong multiplicativity is equivalent to a rigid Boolean phenomenon: localization at $S$ comes from a central idempotent.
Equivalently, the spectral subset
\[
D(S)=\{\mathfrak p \in \Spec(R): \mathfrak p \cap S=\varnothing\}
\]
is clopen.
Once this classification is in place, many auxiliary statements become simpler and more transparent.

The paper also incorporates several prime-theoretic applications that were developed from a different point of view.
In particular, we connect strongly multiplicative sets with strongly prime ideals and strongly zero-dimensional rings, revisit the open question of Hamed and Malek on the necessity of strong multiplicativity for the behavior of chains of $S$-prime ideals \cite{Hamed}, prove a strong form of Krull's separation lemma, and obtain a clean maximal-ideal correspondence for strongly multiplicative localizations.
On the module-theoretic side, we record the associated split torsion theory, the almost multiplicative variant, and a refinement based on the Mittag--Leffler property.
Finally, we relate the theory to the regular $m$-complement operator of Anderson and Chang.

Section~2 contains the structural classification and several permanence results.
Section~3 treats colon ideals, strongly prime ideals, strongly zero-dimensional rings, $\pi$-regular rings and von Neumann regular rings.
Section~4 gives the torsion-theoretic and module-theoretic refinements.
Section~5 studies the open-question examples and $S$-minimal primes.
Section~6 proves the strong Krull separation theorem and the maximal correspondence theorem.
Section~7 treats the regular $m$-complement operator and the total quotient ring case.

\section{Structure and permanence}

We begin with the basic definition and its elementary reformulation.

\begin{definition}\label{definition}
Let $R$ be a commutative ring with $1\neq 0$, and let $S\subseteq R$ be a multiplicative set; that is, $0\notin S$, $1\in S$, and $st\in S$ for all $s,t\in S$. Then:
\begin{enumerate}
\item We say that $S$ is a \textit{strongly multiplicative set} if
$\left(\bigcap_{i\in\Delta} s_iR\right)\cap S\neq\emptyset$
for every family $(s_i)_{i\in\Delta}$ in $S$ \cite{Hamed}.

\item We say that $S$ satisfies the \textit{maximal multiple condition} if there exists $t\in S$ such that $s\mid t$ for every $s\in S$; equivalently, $Rt\subseteq Rs$ for every $s\in S$ \cite{Anderson}. In this case, we say that $t$ is a \emph{least} element of $S$ with respect to inclusion of principal ideals.
\end{enumerate}
\end{definition}

Every multiplicative set $S$ of $R$ satisfies
$\left(\bigcap_{i\in\Delta} s_iR\right)\cap S\neq\emptyset$
whenever $\Delta$ is finite. This property need not hold for infinite index sets. For example, let $R=\mathbb{Z}[X]$ and let $S=\{X^n\mid n\in\mathbb{N}\cup\{0\}\}$. Then $S$ is a multiplicative set, but
$\bigcap_{n\in\mathbb{N}} X^n\mathbb{Z}[X]=0,$
and hence $S$ is not strongly multiplicative.

We now give some examples of strongly multiplicative sets. In fact, each of the following examples satisfies the maximal multiple condition.

\begin{example}\label{ex1}
Let $R$ be a ring. Then:
\begin{enumerate}
\item Every multiplicative set $S\subseteq U(R)$ is strongly multiplicative.

\item Every finite multiplicative set $S$ of $R$ is strongly multiplicative.

\item Let $R$ be a Boolean ring with an ascending chain of nonzero principal ideals
\[
Ra_1\subsetneq Ra_2\subsetneq \cdots \subsetneq Ra_n\subsetneq\cdots.
\]
Set $S=\{a_n\mid n\in\mathbb{N}\}\cup\{1\}$. Since $R$ is Boolean, each element is idempotent. If $i\le j$, then $Ra_i\subseteq Ra_j$, so $a_i=ra_j$ for some $r\in R$, and hence
\[
a_ia_j=ra_j^2=ra_j=a_i.
\]
Thus $S$ is a multiplicative set. Moreover, $a_1\in\bigcap_{n\in\mathbb{N}} Ra_n$, so $(\bigcap_{n\in\mathbb{N}} Ra_n)\cap S\neq\emptyset$. Therefore $S$ is strongly multiplicative.

\item Let $R=(P(\mathbb{N}),\triangle,\cap)$ be the Boolean ring of all subsets of $\mathbb{N}$. For each $n\in\mathbb{N}$, set $A_n=\{1,2,\ldots,n\}$ and let
\[
S=\{A_n\mid n\in\mathbb{N}\}\cup\{\mathbb{N}\}.
\]
Then $A_iA_j=A_i\cap A_j=A_{\min\{i,j\}}$, so $S$ is a multiplicative set. Also,
\[
RA_1\subsetneq RA_2\subsetneq \cdots \subsetneq RA_n\subsetneq\cdots,
\]
and $A_1\in\bigcap_{n\in\mathbb{N}} RA_n$. Hence $(\bigcap_{n\in\mathbb{N}} RA_n)\cap S\neq\emptyset$, so $S$ is strongly multiplicative.
\end{enumerate}
\end{example}

In recent studies on $S$-versions of algebraic structures, some authors assume that $S$ is a strongly multiplicative set in order to prove basic properties of the corresponding notions, while others assume that $S$ satisfies the maximal multiple condition (see, for example, \cite{Baek, Hamed}, etc.). Our first result shows that these two conditions are equivalent.

\begin{proposition}\label{prop:minimal}
Let $S \subseteq R$ be a multiplicative set. Then $S$ is strongly multiplicative if and only if $S$ satisfies the maximal multiple condition.
\end{proposition}

\begin{proof}
Assume that $S$ is strongly multiplicative. Applying the definition to the family consisting of all elements of $S$, we obtain $\left( \bigcap_{s \in S} sR \right) \cap S \neq \emptyset.$
Choose $t \in \left( \bigcap_{s \in S} sR \right) \cap S$. Then $t \in sR$ for every $s \in S$, and hence $s \mid t$ for every $s \in S$. Thus $S$ satisfies the maximal multiple condition.

Conversely, suppose that there exists $t \in S$ such that $s \mid t$ for every $s \in S$. Let $(s_i)_{i \in I}$ be any family of elements of $S$. Since $s_i \mid t$ for every $i \in I$, we have $t \in s_iR$ for every $i \in I$. Therefore,
$t \in \bigcap_{i \in I} s_iR.$
As $t \in S$, it follows that $\left( \bigcap_{i \in I} s_iR \right) \cap S \neq \emptyset.$
Hence $S$ is strongly multiplicative.
\end{proof}

\begin{remark}
If $t, t' \in S$ both satisfy the maximal multiple condition, then $t \mid t'$ and $t' \mid t$. Consequently, $Rt = Rt'$. In this sense, such an element is unique up to generating the same principal ideal.
\end{remark}

We now record a basic property of strongly multiplicative sets that will be used repeatedly in the sequel.

\begin{lemma}\label{lem:idempotent}
Let $S$ be a strongly multiplicative set with least element $t$.
Then there exists an idempotent $e \in R$ such that $tR=eR$.
More precisely, one can write $t=ue$ for a unit $u \in R$.
\end{lemma}

\begin{proof}
Since $t^2 \in S$ and $t$ is least, we have $t \in t^2R$, say $t=t^2a$ for some $a \in R$.
Set $e=ta$.
Then
\[
e^2=t^2a^2=(t^2a)a=ta=e,
\]
so $e$ is idempotent.
Also $e=ta \in tR$, while $t=t^2a=te \in eR$, hence $tR=eR$.

Now define $u=t+1-e$.
Because $t=te$, we have
\[
u(ae+1-e)=(t+1-e)(ae+1-e)=tae+1-e=e+1-e=1,
\]
and similarly $(ae+1-e)u=1$.
Thus $u$ is a unit.
Since $et=t$ and $e(1-e)=0$, we obtain
\[
ue=e(t+1-e)=et=t.
\]
Hence $t=ue$ with $u \in U(R)$.
\end{proof}

The next theorem is the basic classification theorem.
It merges the intersection property, the least-element condition, the idempotent description of the localization, and the spectral picture.

\begin{theorem}\label{thm:main}
Let $S \subseteq R$ be a multiplicative set.
The following are equivalent.
\begin{enumerate}
\item $S$ is strongly multiplicative.
\item For every family of ideals $(I_\lambda)_{\lambda \in \Lambda}$ of $R$,
\[
\Bigl(\bigcap_{\lambda \in \Lambda} I_\lambda\Bigr)_S
=
\bigcap_{\lambda \in \Lambda} (I_\lambda)_S
\qquad\text{inside } R_S.
\]
\item There exists an idempotent $e \in R$ such that $R_S \cong Re$ as rings.
\item There exists an idempotent $e \in R$ such that $D(S)=D(e)$.
Equivalently, $D(S)$ is clopen in $\Spec(R)$.
\end{enumerate}
Moreover, if $t$ is a least element of $S$ and $t=ue$ is as in Lemma~\ref{lem:idempotent}, then
\[
R_S \cong R_t \cong R_e \cong Re.
\]
\end{theorem}

\begin{proof}
(1) $\Rightarrow$ (2).
Let $t \in S$ be a least element.
The inclusion
\[
\Bigl(\bigcap_{\lambda \in \Lambda} I_\lambda\Bigr)_S
\subseteq
\bigcap_{\lambda \in \Lambda} (I_\lambda)_S
\qquad\text{inside } R_S.
\]
always holds.
Conversely, let $a/b$ belong to the right-hand side.
For each $\lambda$ there exists $s_\lambda \in S$ such that $s_\lambda a \in I_\lambda$.
Since $s_\lambda \mid t$, we also have $ta \in I_\lambda$ for every $\lambda$.
Thus $ta \in \bigcap_\lambda I_\lambda$, and therefore
\[
\frac{a}{b}=\frac{ta}{tb}\in
\Bigl(\bigcap_{\lambda \in \Lambda} I_\lambda\Bigr)_S.
\]

(2) $\Rightarrow$ (1).
Apply (2) to the family of principal ideals $(s_iR)_{i \in I}$ arising from a family $(s_i)_{i \in I}$ in $S$.
Since $1 \in (s_iR)_S$ for each $i$, we get
\[
1 \in \bigcap_{i \in I} (s_iR)_S=\Bigl(\bigcap_{\lambda \in \Lambda} I_\lambda\Bigr)_S.
\]
Hence there exists $t \in S$ with $t \in \bigcap_{i \in I} s_iR$, which proves that $S$ is strongly multiplicative.

(1) $\Rightarrow$ (3).
Let $t \in S$ be a least element.
Because every element of $S$ divides $t$, every element of $S$ becomes invertible in $R_t$; thus the universal property of localization yields $R_S \cong R_t$.
By Lemma~\ref{lem:idempotent}, we may write $t=ue$ with $e^2=e$ and $u$ a unit.
Since $u$ is invertible, $R_t \cong R_e$.
Finally, if $e$ is idempotent, the map
\[
R_e \longrightarrow Re,\qquad \frac{r}{e^n}\longmapsto re
\]
is a ring isomorphism.
Hence $R_S \cong Re$.

(3) $\Rightarrow$ (2).
Assume $R_S \cong Re$ for an idempotent $e$.
Extension of an ideal $I \subseteq R$ to $Re$ is simply $Ie$.
Therefore, for any family $(I_\lambda)$,
$\Bigl(\bigcap_{\lambda} I_\lambda\Bigr)e
\subseteq
\bigcap_{\lambda} I_\lambda e.$
For the reverse inclusion, let $x$ belong to the right-hand side.
Then $x \in I_\lambda e \subseteq I_\lambda$ for every $\lambda$, so $x \in \bigcap_\lambda I_\lambda$.
Since also $x=xe$, we conclude that
$x \in \Bigl(\bigcap_{\lambda} I_\lambda\Bigr)e.$
Hence
$\Bigl(\bigcap_{\lambda} I_\lambda\Bigr)e
=
\bigcap_{\lambda} I_\lambda e,$
which is exactly (2) after identifying $R_S$ with $Re$.

(3) $\Rightarrow$ (4).
For every localization one has $\Spec(R_S)\cong D(S)$.
If $R_S \cong Re$, then $\Spec(R_S)\cong \Spec(Re)\cong D(e)$.
Thus $D(S)=D(e)$, and $D(e)$ is clopen because $V(e)=D(1-e)$.

(4) $\Rightarrow$ (3).
Assume $D(S)=D(e)$ for an idempotent $e$.
For each $s \in S$ we have $D(S)\subseteq D(s)$, hence $D(e)\subseteq D(s)$, so $s$ becomes invertible in $R_e$.
Thus the localization map $R \to R_e$ factors uniquely through $R_S$.
Conversely, $D(S)=D(e)$ implies that $e$ becomes invertible in $R_S$, so the localization map $R \to R_S$ factors uniquely through $R_e$.
The two induced maps are inverse to one another by uniqueness.
Hence $R_S \cong R_e \cong Re$.
\end{proof}

\begin{remark}
Theorem~\ref{thm:main}(2) shows that if $S$ is a strongly multiplicative set of $R$, then the localization map
$R \longrightarrow R_S$ is intersection flat for ideals in the sense of Hochster and Jeffries (\cite{HochsterJeffries}); that is, for every family of ideals $(I_\lambda)_{\lambda\in\Lambda}$ of $R$, one has
$\left(\bigcap_{\lambda\in\Lambda} I_\lambda\right)R_S = \bigcap_{\lambda\in\Lambda} I_\lambda R_S$ inside $R_S$. Thus, strongly multiplicative localizations form a distinguished class of ring homomorphisms for which arbitrary intersections of ideals commute with extension. Moreover, by Theorem~\ref{thm:main}(3), there exists an idempotent $e\in R$ such that $R_S \cong Re.$
Hence these localizations are precisely those arising from passage to an idempotent summand of $R$, and the intersection-flatness-for-ideals property may be viewed as a consequence of this idempotent decomposition. In general, this is weaker than full intersection flatness, which requires the analogous commutation property for arbitrary families of submodules of finitely generated modules.
\end{remark}

\begin{corollary}\label{cor:indecomp}
If $R$ is indecomposable, then every strongly multiplicative set of $R$ is contained in $U(R)$.
In particular, every strongly multiplicative set of a local ring is contained in $U(R)$.
\end{corollary}

\begin{proof}
If $S$ is strongly multiplicative, Theorem~\ref{thm:main} gives $R_S \cong Re$ for an idempotent $e$.
Since $R$ is indecomposable, $e=1$.
Thus $R_S \cong R$, and this happens exactly when every element of $S$ is a unit.
The final assertion follows because every local ring is indecomposable.
\end{proof}

\begin{example}[A domain does not provide new examples]\label{ex:domain}
Let $R=\mathbb{Z}$ and $S=\mathbb{Z}\setminus\{0\}$.
Then $R_S=\mathbb{Q}$, and the example from the introduction shows that localization at $S$ does not commute with arbitrary intersections of ideals.
Hence $S$ is not strongly multiplicative.
Equivalently, $D(S)$ is not clopen in $\Spec(\mathbb{Z})$.
\end{example}

\begin{example}[Finite products]\label{ex:product}
Let $R=R_1 \times R_2$ and let $e=(1,0)$.
Then $Re \cong R_1$ and $D(e)=\Spec(R_1) \subseteq \Spec(R_1) \sqcup \Spec(R_2)=\Spec(R)$.
Therefore every multiplicative set $S$ with $D(S)=D(e)$ is strongly multiplicative, and the associated localization is just the projection $R_1 \times R_2 \to R_1$.
This is the model example for a nontrivial strongly multiplicative localization.
\end{example}

The least-element description makes several permanence properties immediate.

\begin{proposition}\label{prop:hom}
Let $f \colon A \to B$ be a ring homomorphism.
\begin{enumerate}
\item If $S$ is a strongly multiplicative set of $A$, $0 \notin f(S)$, and $1 \in f(S)$, then $f(S)$ is a strongly multiplicative set of $B$.
\item If $S$ is a strongly multiplicative set of $A$, $f$ is surjective, and $0 \notin f(S)$, then $f(S)$ is a strongly multiplicative set of $B$.
\end{enumerate}
\end{proposition}

\begin{proof}
Let $t \in S$ be a least element.
For any $s \in S$ we have $s \mid t$, say $t=as$.
Applying $f$ gives $f(t)=f(a)f(s)$, so every element of $f(S)$ divides $f(t)$.
Thus $f(t)$ is a least element of $f(S)$, provided that $0 \notin f(S)$ and $1 \in f(S)$.
This proves (1), and (2) follows because surjectivity implies $1=f(1)\in f(S)$.
\end{proof}

\begin{corollary}\label{cor:factor}
Let $I$ be an ideal of $R$, and let $S$ be a strongly multiplicative set such that $I \cap S=\varnothing$.
Then
\[
\overline S=\{s+I : s \in S\}
\]
is a strongly multiplicative set of $R/I$.
\end{corollary}

\begin{proof}
Apply Proposition~\ref{prop:hom} to the quotient map $R \to R/I$.
The condition $I \cap S=\varnothing$ guarantees that $0 \notin \overline S$.
\end{proof}

\begin{theorem}\label{thm:product}
Let $R_i$ be a ring and $S_i$ a multiplicative set of $R_i$ for $i=1,2$.
Set $R=R_1 \times R_2$ and $S=S_1 \times S_2$.
Then $S$ is strongly multiplicative if and only if both $S_1$ and $S_2$ are strongly multiplicative.
\end{theorem}

\begin{proof}
Assume first that $S$ is strongly multiplicative.
The coordinate projections $R \to R_i$ are surjective and send $S$ onto $S_i$, so Proposition~\ref{prop:hom} shows that each $S_i$ is strongly multiplicative.

Conversely, assume that each $S_i$ is strongly multiplicative, and choose least elements $t_i \in S_i$.
Then $t=(t_1,t_2)$ belongs to $S$.
For any $s=(s_1,s_2)\in S$, each $s_i$ divides $t_i$, so $s$ divides $t$.
Hence $t$ is a least element of $S$.
\end{proof}

\begin{remark}
A strongly multiplicative set of a product ring need not be a Cartesian product of strongly multiplicative sets. Indeed, let $0\neq e\in R_1$ be an idempotent and set
\[
S=\{(e,0),(e,1),(1,1)\}\subseteq R_1\times R_2.
\]
Then $S$ is a finite multiplicative set, and hence it is strongly multiplicative by Example~\ref{ex1}. If $e=1$, then
\[
S=\{1\}\times \{0,1\},
\]
but $\{0,1\}$ is not a multiplicative set of $R_2$, and therefore it is not strongly multiplicative. If $e\neq 1$, then the projections of $S$ onto the two factors are $\{e,1\}$ and $\{0,1\}$, respectively. Hence, if $S=S_1\times S_2$ for some subsets $S_1\subseteq R_1$ and $S_2\subseteq R_2$, then necessarily $S_1=\{e,1\}$ and $S_2=\{0,1\}$, so $S$ would have four elements, a contradiction. Therefore, $S$ cannot be written as $S_1\times S_2$ for any multiplicative sets $S_1$ of $R_1$ and $S_2$ of $R_2$.
\end{remark}

\begin{proposition}\label{prop:locperm}
Let $S$ and $T$ be multiplicative sets of $R$ with $0 \notin ST=\{st : s \in S,\ t \in T\}$.
\begin{enumerate}
\item If $S$ and $T$ are strongly multiplicative, then $ST$ is strongly multiplicative.
\item If $T$ is strongly multiplicative, then $T_S:=\{t/s : s \in S,\ t \in T\}$ is a strongly multiplicative set of $R_S$.
\end{enumerate}
\end{proposition}

\begin{proof}
If $s^\ast$ and $t^\ast$ are least elements of $S$ and $T$, then $s^\ast t^\ast$ is a least element of $ST$.
This proves (1).

For (2), let $t$ be a least element of $T$.
Given $z/s \in T_S$, the divisibility $z \mid t$ in $R$ yields $z/s \mid t/1$ in $R_S$.
Hence $t/1$ is a least element of $T_S$.
\end{proof}

The following example shows that the converse of Proposition~\ref{prop:locperm}(2) does not hold in general.

\begin{example}
Let $R=k[X]$ be the polynomial ring over a field $k$, and let
\[
S=T=\{X^n \mid n\in \mathbb{N}\cup\{0\}\}.
\]
Since $\bigcap_{n\in\mathbb{N}} X^nR=0,$ it follows that $T$ is not a strongly multiplicative set of $R$. On the other hand, $R_S=k[X,X^{-1}]$ is the Laurent polynomial ring, and
\[
T_S=\{X^n \mid n\in\mathbb{Z}\}
\]
consists entirely of units in $R_S$. Therefore, by Example~\ref{ex1}, $T_S$ is a strongly multiplicative set of $R_S$.
\end{example}

Let $R$ be a ring, let $I$ and $J$ be ideals of $R$, and let $S$ be a multiplicative set of $R$. Then it is easy to see that
\[
(I:J)_S\subseteq (I_S:J_S).
\]
Moreover, if $J$ is finitely generated, then
\[
(I:J)_S=(I_S:J_S).
\]
However, if $J$ is not finitely generated, this equality need not hold in general, as the following example shows.

\begin{example}\label{examplecolon}
Let $A=\mathbb{Z}[X_1,X_2,\ldots]$ be the polynomial ring in countably many indeterminates, and let
\[
Q=(2X_1,2^2X_2,\ldots,2^nX_n,\ldots)
\]
be an ideal of $A$. Set $R=A/Q$ and
\[
S=\{2^n+Q \mid n\in \mathbb{N}\cup\{0\}\}.
\]
Then $S$ is a multiplicative set of $R$. Let
\[
I=Q/Q=0
\qquad\text{and}\qquad
J=(X_1,X_2,\ldots)/Q.
\]

We claim that $(I:J)=0$. Indeed, let $f+Q\in (I:J)=(0:J)$. Then $fX_n\in Q$ for every $n\geq 1$. If $f\notin Q$, then some monomial term of $f$ does not belong to $Q$. Choose such a term $cX_1^{a_1}\cdots X_m^{a_m}.$
Then $2^i\nmid c$ for every $i$ with $a_i>0$. Choosing $n>m$ with $2^n\nmid c$, the monomial $cX_1^{a_1}\cdots X_m^{a_m}X_n$ appears in $fX_n$ and does not belong to $Q$, a contradiction. Hence $f\in Q$, so $(I:J)=0$. Therefore,
$(I:J)_S=0.$

On the other hand, for each $n\geq 1$ we have $(2^n+Q)(X_n+Q)=0$ in $R$. Since $2^n+Q\in S$, it follows that $X_n/1=0$ in $R_S$ for every $n$. Hence $J_S=0$. Also, $I_S=0$, and therefore
\[
(I_S:J_S)=(0:0)=R_S.
\]
Consequently,
\[
(I:J)_S\neq (I_S:J_S).
\]
\end{example}

The next theorem records a useful consequence for colon ideals.

\begin{theorem}\label{thm:colon}
Let $S$ be a strongly multiplicative set of $R$, and let $I,J$ be ideals of $R$.
Then
\[
(I:J)_S=(I_S:J_S).
\]
No finite generation hypothesis on $J$ is needed.
\end{theorem}

\begin{proof}
The inclusion
\[
(I:J)_S\subseteq (I_S:J_S)
\]
always holds.
For the reverse inclusion, let $a/b \in (I_S:J_S)$.
Then for every $x \in J$ one has $(a/b)(x/1)\in I_S$, so there exists $s_x \in S$ with $s_xax \in I$.
Let $t$ be a least element of $S$.
Since $s_x \mid t$, it follows that $tax \in I$ for every $x \in J$.
Thus $ta \in (I:J)$, and therefore
\[
\frac{a}{b}=\frac{ta}{tb}\in (I:J)_S. \qedhere
\]
\end{proof}

Two standard constructions behave well with respect to strong multiplicativity.

\begin{theorem}\label{thm:trivial}
Let $M$ be an $R$-module, let $N$ be a submodule of $M$, and write $R \propto M$ for the trivial extension.
Then $S \propto N$ is a strongly multiplicative set of $R \propto M$ if and only if $S$ is a strongly multiplicative set of $R$.
\end{theorem}

\begin{proof}
If $S \propto N$ is strongly multiplicative, then the projection $R \propto M \to R$ maps it onto $S$, so Proposition~\ref{prop:hom} shows that $S$ is strongly multiplicative.

Conversely, assume that $S$ is strongly multiplicative, and let $t \in S$ be least.
If $t$ is a unit of $R$, then every element of $S$ is a unit, and therefore every element of $S \propto N$ is a unit of $R \propto M$.
Hence $S \propto N$ is strongly multiplicative.

Assume now that $t$ is not a unit.
By Lemma~\ref{lem:idempotent}, $t=ue$ for a unit $u$ and a nontrivial idempotent $e$.
After identifying
\[
R \cong Re \times R(1-e)
\qquad\text{and}\qquad
M \cong eM \times (1-e)M,
\]
we may write
\[
e=(1,0),\qquad
t=(u_1,0),
\qquad
s=(s_1,s_2)
\]
for $s \in S$, with $s_1$ a unit because $s \mid t$.
Fix $(s,n)\in S \propto N$, where $n=(n_1,n_2)$.
Choose
\[
a=(u_1s_1^{-1},0)
\qquad\text{and}\qquad
m=\bigl(-s_1^{-1}a_1n_1,0\bigr),
\]
where $a_1=u_1s_1^{-1}$.
Then
\[
sa=t
\qquad\text{and}\qquad
sm+an=0.
\]
Hence $(s,n)(a,m)=(t,0),$ so every element of $S \propto N$ divides $(t,0)$.
Therefore $(t,0)$ is a least element of $S \propto N$.
\end{proof}

\begin{theorem}\label{thm:amalgam}
Let $f \colon A \to B$ be a ring homomorphism, let $J$ be an ideal of $B$, and let
\[
A \bowtie^f J=\{(a,f(a)+j): a \in A,\ j \in J\}
\]
be the amalgamated algebra along $J$ \cite{Danna}.
For a multiplicative set $S \subseteq A$, define
\[
S'=\{(s,f(s)) : s \in S\} \cup \{(1,1)\}\subseteq A \bowtie^f J.
\]
Then $S'$ is strongly multiplicative if and only if $S$ is strongly multiplicative.
\end{theorem}

\begin{proof}
If $S'$ is strongly multiplicative, then the first projection $A \bowtie^f J \to A$ maps $S'$ onto $S$, so Proposition~\ref{prop:hom} gives the result.

Conversely, let $t \in S$ be a least element.
For any $(s,f(s)) \in S'$, the divisibility $s \mid t$ in $A$ yields
\[
(t,f(t))=(s,f(s))(a,f(a))
\]
for some $a \in A$.
Thus $(t,f(t))$ is a least element of $S'$.
\end{proof}

Strongly multiplicative sets are also constrained by the Jacobson radical.

\begin{proposition}\label{prop:jac}
If $S$ is a strongly multiplicative set of $R$, then $S \cap \Jac(R)=\varnothing$. In particular, if $R$ is a local ring, then every strongly multiplicative set of $R$ is contained in $U(R)$.
\end{proposition}

\begin{proof}
Assume, for contradiction, that $x \in S \cap \Jac(R)$, and let $t \in S$ be a least element with respect to inclusion of principal ideals. Since $tx \in S$, the defining property of $t$ yields $t \in txR$. Thus there exists $a \in R$ such that
$t=txa.$
Hence, $t(1-xa)=0.$
Since $x \in \Jac(R)$, the element $1-xa$ is a unit. Therefore, $t=0$, which contradicts the fact that $0\notin S$.

Thus $S \cap \Jac(R)=\varnothing$. If $R$ is local, then $\Jac(R)=R\setminus U(R)$. Hence $S\subseteq U(R)$.
\end{proof}

\begin{example}
By Proposition~\ref{prop:jac}, every strongly multiplicative set of a ring $R$ is disjoint from $\Jac(R)$. However, this need not hold for a general multiplicative set. Let $R=\mathbb{Z}_{(p)}\cap \mathbb{Z}_{(q)}$, where $p$ and $q$ are distinct prime numbers. Then $R$ has exactly two maximal ideals, $pR$ and $qR$, and hence
\[
\Jac(R)=pR\cap qR=pqR.
\]
Set $S=(pqR\setminus\{0\})\cup\{1\}$. Then $S$ is a multiplicative set of $R$, and
\[
S\cap \Jac(R)=pqR\setminus\{0\}\neq\emptyset.
\]
Thus $S$ is not strongly multiplicative.
\end{example}

The last structural result describes the saturation of a strongly multiplicative set.

\begin{theorem}\label{thm:sat}
Let $S$ be a strongly multiplicative set of $R$, and let $\overline S$ denote its saturation.
Choose an idempotent $e \in R$ with $R_S \cong Re$.
Then
\[
\overline S=\{r \in R : re \in U(Re)\}.
\]
In particular:
\begin{enumerate}
\item if $R$ is indecomposable, then $\overline S=U(R)$;
\item if $R=R_1 \times R_2$ and, under this decomposition, $e$ equals $(1,0)$, $(0,1)$, or $(1,1)$, then $\overline S$ equals $U(R_1)\times R_2$, $R_1 \times U(R_2)$, or $U(R_1)\times U(R_2)$, respectively.
\end{enumerate}
Moreover, $\overline S$ is strongly multiplicative.
\end{theorem}

\begin{proof}
For any multiplicative set, one has
\[
r \in \overline S
\quad\Longleftrightarrow\quad
r/1 \in U(R_S).
\]
Indeed, if $r \mid s$ for some $s \in S$, then $r/1$ is invertible in $R_S$.
Conversely, if $(r/1)(a/s)=1$ in $R_S$, then there exists $u \in S$ such that $u(ra-s)=0$, so $ura=us \in S$, and hence $r \in \overline S$.

After identifying $R_S$ with $Re$, the above criterion becomes
\[
r \in \overline S
\quad\Longleftrightarrow\quad
re \in U(Re).
\]
This proves the displayed formula.

If $R$ is indecomposable, Corollary~\ref{cor:indecomp} shows that $e=1$, and then the formula gives $\overline S=U(R)$.
The product descriptions are immediate from the same formula.

Finally, $\overline S$ is strongly multiplicative because it is the preimage of the unit group of $Re$ under the map $R \to Re$.
Equivalently, in the three explicit cases above, a least element is given by $1$, $(1,0)$, or $(0,1)$.
\end{proof}

\section{Intersections, strongly prime ideals, and strongly zero-dimensional rings}

We next turn to prime ideals.
Recall from \cite{Tekir} that a prime ideal $P$ is called \emph{strongly prime} if for every family of ideals $(I_\lambda)_{\lambda \in \Lambda}$ with
$\bigcap_{\lambda \in \Lambda} I_\lambda \subseteq P,$
one has $I_\mu \subseteq P$ for some index $\mu$.

\begin{theorem}\label{thm:strongprime}
Let $P$ be a prime ideal of $R$.
Then $P$ is strongly prime if and only if $R \setminus P$ is a strongly multiplicative set.
\end{theorem}

\begin{proof}
Assume first that $P$ is strongly prime, and let $(s_i)_{i \in I}$ be a family in $R \setminus P$.
If
\[
\Bigl(\bigcap_{i \in I} s_iR\Bigr)\cap (R \setminus P)=\varnothing,
\]
then $\bigcap_i s_iR \subseteq P$.
Since $P$ is strongly prime, there exists $j$ with $s_jR \subseteq P$, so $s_j \in P$, a contradiction.
Therefore $R \setminus P$ is strongly multiplicative.

Conversely, assume that $R \setminus P$ is strongly multiplicative.
Let $(I_\lambda)_{\lambda \in \Lambda}$ be a family of ideals with $\bigcap_\lambda I_\lambda \subseteq P$.
If no $I_\lambda$ is contained in $P$, choose $a_\lambda \in I_\lambda \setminus P$ for each $\lambda$.
Then each $a_\lambda$ lies in $R \setminus P$, so strong multiplicativity yields an element
\[
s \in \Bigl(\bigcap_{\lambda \in \Lambda} a_\lambda R\Bigr)\cap (R \setminus P).
\]
But $\bigcap_\lambda a_\lambda R \subseteq \bigcap_\lambda I_\lambda \subseteq P$, which gives $s \in P$, a contradiction.
Hence some $I_\lambda$ is contained in $P$, and $P$ is strongly prime.
\end{proof}

\begin{corollary}\label{cor:prime-localization}
Let $P$ be a prime ideal of $R$.
The following are equivalent.
\begin{enumerate}
\item $P$ is strongly prime.
\item For every family of ideals $(I_\lambda)_{\lambda \in \Lambda}$,
\[
\Bigl(\bigcap_{\lambda \in \Lambda} I_\lambda\Bigr)_P
=
\bigcap_{\lambda \in \Lambda} (I_\lambda)_P.
\]
\item $D(P)=\Spec(R)\setminus V(P)$ is clopen in $\Spec(R)$.
\end{enumerate}
\end{corollary}

\begin{proof}
Apply Theorem~\ref{thm:strongprime} to the multiplicative set $R \setminus P$, and then invoke Theorem~\ref{thm:main}.
\end{proof}

Following \cite{Tekir}, a ring is called \emph{strongly zero-dimensional} if every prime ideal is strongly prime.
Combining Theorem~\ref{thm:strongprime}, Corollary~\ref{cor:prime-localization}, and Gottlieb's classification \cite[Theorem~2.4]{Gotlieb}, we obtain the following reformulation.

\begin{corollary}\label{cor:szd}
For a ring $R$, the following are equivalent.
\begin{enumerate}
\item $R$ is strongly zero-dimensional.
\item For every prime ideal $P$, the complement $R \setminus P$ is strongly multiplicative.
\item For every prime ideal $P$, localization at $P$ commutes with arbitrary intersections of ideals.
\item $R$ is zero-dimensional and quasi-semilocal.
\end{enumerate}
\end{corollary}

\begin{proof}
The equivalence of (1), (2), and (3) follows from Theorem~\ref{thm:strongprime} and Corollary~\ref{cor:prime-localization}.
The equivalence of (1) and (4) is exactly Gottlieb's theorem \cite[Theorem~2.4]{Gotlieb}.
\end{proof}

Following \cite{McCoy}, a ring $R$ is called a $\pi$-\emph{regular ring} if for each $a\in R$ there exist $x\in R$ and $n\in\mathbb{N}$ such that $a^n=xa^{2n}$. Equivalently, for each $a\in R$, the ideal $a^nR$ is idempotent for some $n\in\mathbb{N}$. If one can always take $n=1$, then $R$ is called a \emph{von Neumann regular ring} \cite{VonNeumann}. It is well known that a ring $R$ is $\pi$-regular if and only if it is zero-dimensional; moreover, $R$ is von Neumann regular if and only if it is reduced and $\pi$-regular (see \cite[Theorem 2.24 and Theorem 3.6]{Koc2020}). The following theorem and its immediate consequence provide new characterizations of $\pi$-regular and von Neumann regular rings in terms of strongly multiplicative sets.

\begin{theorem}
Let $R$ be a ring. Then $R$ is a $\pi$-regular ring if and only if $S_a=\{a^n \mid n\in\mathbb{N}\cup\{0\}\}$ is a strongly multiplicative set of $R$ for each non-nilpotent element $a\in R$.
\end{theorem}

\begin{proof}
$(\Leftarrow)$ Let $a\in R$ be a non-nilpotent element. By assumption, $S_a=\{a^n \mid n\in\mathbb{N}\cup\{0\}\}$ is a strongly multiplicative set of $R$. Thus, there exists a least element $t=a^n\in S_a$ with respect to inclusion of principal ideals, for some $n\in\mathbb{N}$. By Lemma~\ref{lem:idempotent}, we may write $a^n=ue$ for some idempotent $e\in R$ and some unit $u\in R$. Hence $a^{2n}=u^2e^2=u^2e=ua^n,$ and therefore $a^n=u^{-1}a^{2n}.$
This shows that $R$ is a $\pi$-regular ring.

$(\Rightarrow)$ Conversely, assume that $R$ is a $\pi$-regular ring, and let $a\in R$ be a non-nilpotent element. Then $S_a=\{a^n \mid n\in\mathbb{N}\cup\{0\}\}$ is a multiplicative set of $R$. Since $R$ is $\pi$-regular, there exists $n\in\mathbb{N}$ such that $a^nR$ is an idempotent ideal. By Lemma~\ref{lem:idempotent}, we have $a^n=ue$ for some idempotent $e\in R$ and some unit $u\in R$. It follows that for every $k\geq n$, $Ra^k=Ra^n.$
Hence $a^n$ is a least element of $S_a$ with respect to inclusion of principal ideals. Therefore $S_a$ is a strongly multiplicative set of $R$.
\end{proof}

\begin{corollary}
Let $R$ be a ring. Then $R$ is a von Neumann regular ring if and only if $R$ is reduced and $S_a=\{a^n \mid n\in\mathbb{N}\cup\{0\}\}$ is a strongly multiplicative set of $R$ for each $0\neq a\in R$.
\end{corollary}

\section{Torsion theory and related refinements}

The classification theorem can be reformulated entirely in torsion-theoretic language.
For a multiplicative set $S$, write
\[
t_S(M)=\{m \in M : sm=0 \text{ for some } s \in S\}
=
\ker(M \to M_S),
\]
and define
\[
\mathcal T_S=\{M : t_S(M)=M\},
\qquad
\mathcal F_S=\{M : t_S(M)=0\},
\]
together with the Gabriel filter
\[
\mathfrak F_S=\{I \subseteq R : I \cap S \neq \varnothing\}.
\]

\begin{theorem}\label{thm:torsion}
Let $S$ be a multiplicative set of $R$.
The following are equivalent.
\begin{enumerate}
\item $S$ is strongly multiplicative.
\item There exists an idempotent $e \in R$ such that
\[
\mathfrak F_S=\{I \subseteq R : e \in I\}.
\]
\item There exists an idempotent $e \in R$ such that, for every $R$-module $M$,
\[
M_S \cong eM
\qquad\text{and}\qquad
t_S(M)=(1-e)M.
\]
\end{enumerate}
When these conditions hold,
\[
\mathcal T_S=\{M : eM=0\},
\qquad
\mathcal F_S=\{M : (1-e)M=0\},
\]
and every $R$-module splits functorially as
\[
M=eM \oplus (1-e)M.
\]
In particular, the hereditary torsion theory attached to $S$ is split.
\end{theorem}

\begin{proof}
Assume that $S$ is strongly multiplicative, let $t \in S$ be least, and write $t=ue$ as in Lemma~\ref{lem:idempotent}.

Let $I$ be an ideal of $R$.
If $I \cap S \neq \varnothing$, choose $s \in I \cap S$.
Since $s \mid t$, we get $t \in I$, hence $e=u^{-1}t \in I$.
Conversely, if $e \in I$, then $t=ue \in I \cap S$.
Therefore
\[
\mathfrak F_S=\{I \subseteq R : e \in I\},
\]
so (1) implies (2).

Assume (2).
For a prime ideal $\mathfrak p$,
\[
\mathfrak p \cap S=\varnothing
\quad\Longleftrightarrow\quad
\mathfrak p \notin \mathfrak F_S
\quad\Longleftrightarrow\quad
e \notin \mathfrak p.
\]
Hence $D(S)=D(e)$, and Theorem~\ref{thm:main} gives (1).
Thus (2) implies (1).

Assume again that $S$ is strongly multiplicative.
By Theorem~\ref{thm:main}, $R_S \cong Re$.
For every $R$-module $M$,
\[
M_S
\cong
M \otimes_R R_S
\cong
M \otimes_R Re
\cong
eM.
\]
Under this identification, the localization map $M \to M_S$ becomes $m \mapsto em$, so
\[
t_S(M)=\ker(m \mapsto em)=(1-e)M.
\]
This proves (1) implies (3), and the displayed descriptions of $\mathcal T_S$ and $\mathcal F_S$ are immediate.

Finally, assume (3).
An ideal $I$ belongs to $\mathfrak F_S$ if and only if $R/I$ is $S$-torsion, that is, if and only if $t_S(R/I)=R/I$.
By (3), this is equivalent to
\[
(1-e)(R/I)=R/I,
\]
which holds if and only if $e \in I$.
Hence (3) implies (2).
\end{proof}

\begin{remark}
Theorem~\ref{thm:torsion} shows that strongly multiplicative sets are exactly the multiplicative sets whose localization torsion theory is defined by a central idempotent.
Thus localization is literally projection onto a direct summand.
\end{remark}

We next incorporate the almost multiplicative variant.

\begin{definition}\cite{BaekLim}
A nonempty subset $S \subseteq R$ is \emph{almost multiplicative} if for every $a,b \in S$ there exist integers $m,n \geq 1$ such that $a^mb^n \in S$.
We write $\langle S\rangle$ for the multiplicative subset generated by $S$.
\end{definition}

\begin{theorem}\label{thm:almost}
Let $S \subseteq R$ be almost multiplicative, and let $R_S^{\mathrm{am}}$ be the quotient ring attached to $S$ in the sense of Baek and Lim.
Then the canonical map $R_S^{\mathrm{am}} \longrightarrow R_{\langle S\rangle}$ is an isomorphism \cite[Theorem~2.5]{BaekLim}.
Consequently, the following are equivalent.
\begin{enumerate}
\item For every family of ideals $(I_\lambda)_{\lambda \in \Lambda}$,
\[
\Bigl(\bigcap_{\lambda \in \Lambda} I_\lambda\Bigr)R_S^{\mathrm{am}}
=
\bigcap_{\lambda \in \Lambda} I_\lambda R_S^{\mathrm{am}}.
\]
\item The multiplicative hull $\langle S\rangle$ is strongly multiplicative.
\item There exists an idempotent $e \in R$ such that $R_S^{\mathrm{am}} \cong Re$.
\item The set
\[
\{\mathfrak p \in \Spec(R) : \mathfrak p \cap S=\varnothing\}
\]
is clopen in $\Spec(R)$.
\end{enumerate}
\end{theorem}

\begin{proof}
By Baek and Lim, $R_S^{\mathrm{am}} \cong R_{\langle S\rangle}$.
Also, for a prime ideal $\mathfrak p$,
\[
\mathfrak p \cap S=\varnothing
\quad\Longleftrightarrow\quad
\mathfrak p \cap \langle S\rangle=\varnothing.
\]
Indeed, if $\mathfrak p$ meets $\langle S\rangle$, then some finite product of elements of $S$ lies in $\mathfrak p$, and primality forces $\mathfrak p$ to meet $S$.
Thus the spectral subset defined by $S$ is exactly $D(\langle S\rangle)$.
Now apply Theorem~\ref{thm:main} to the multiplicative set $\langle S\rangle$.
\end{proof}

Let $R$ be a ring and $M$ a right $R$-module. We say that $M$ is a \emph{Mittag--Leffler module} if for every family $\{Q_i\}_{i\in I}$ of left $R$-modules, the canonical map
\[
M\otimes_R \prod_{i\in I} Q_i \longrightarrow \prod_{i\in I}(M\otimes_R Q_i)
\]
is injective.

The next result is taken from the work of Datta, Epstein, and Tucker \cite{DET}.

\begin{theorem}[Datta--Epstein--Tucker]\label{thm:DET}
Let $M$ be a flat $R$-module.
The following are equivalent.
\begin{enumerate}
\item $M$ is a Mittag--Leffler module.
\item For every finitely generated $R$-module $L$ and every family of submodules $(N_\lambda)_{\lambda \in \Lambda}$ of $L$,
\[
M \otimes_R \Bigl(\bigcap_{\lambda \in \Lambda} N_\lambda\Bigr)
=
\bigcap_{\lambda \in \Lambda} \bigl(M \otimes_R N_\lambda\bigr)
\qquad\text{inside } M \otimes_R L.
\]
\end{enumerate}
\end{theorem}

\begin{corollary}\label{cor:ML}
Let $S$ be a multiplicative set of $R$.
If $R_S$ is a Mittag--Leffler $R$-module, then for every finitely generated $R$-module $L$ and every family of submodules $(N_\lambda)_{\lambda \in \Lambda}$ of $L$,
\[
\Bigl(\bigcap_{\lambda \in \Lambda} N_\lambda\Bigr)_S
=
\bigcap_{\lambda \in \Lambda} (N_\lambda)_S
\qquad\text{inside } L_S.
\]
In particular, $S$ is strongly multiplicative.
\end{corollary}

\begin{proof}
Localization is tensoring with $R_S$.
Since $R_S$ is flat and Mittag--Leffler, Theorem~\ref{thm:DET} applies and gives
\[
R_S \otimes_R \Bigl(\bigcap_{\lambda \in \Lambda} N_\lambda\Bigr)
=
\bigcap_{\lambda \in \Lambda} \bigl(R_S \otimes_R N_\lambda\bigr).
\]
This is exactly the displayed equality.
Taking $L=R$ recovers the ideal case, so Theorem~\ref{thm:main} shows that $S$ is strongly multiplicative.
\end{proof}

\section{The open question on chains of \texorpdfstring{$S$}{S}-prime ideals}

We now return to the $S$-prime ideal theory of Hamed and Malek \cite{Hamed}.
Recall that an ideal $P$ with $P \cap S=\varnothing$ is called \emph{$S$-prime} if there exists $s \in S$ such that for all $a,b \in R$,
\[
ab \in P
\quad\Longrightarrow\quad
sa \in P \ \text{or}\ sb \in P.
\]
An $S$-prime ideal minimal among those containing a fixed ideal $I$ is called a \emph{minimal $S$-prime over $I$}; when $I=0$, we simply speak of an \emph{$S$-minimal prime ideal}.

Hamed and Malek proved that if $S$ is strongly multiplicative, then the intersection of a chain of $S$-prime ideals is again $S$-prime and minimal $S$-primes exist over every ideal disjoint from $S$ \cite[Proposition~5]{Hamed}.
They also asked whether the strong multiplicativity of $S$ is really necessary.
The next examples show that the answer is subtle: the condition is necessary in general, but not in every special case.

\begin{example}[Failure without strong multiplicativity]\label{ex:counter1}
Let
\[
A=\mathbb{Z}[X_1,X_2,\ldots],
\qquad
Q=(2X_1,2^2X_2,\ldots,2^nX_n,\ldots),
\qquad
R=A/Q,
\]
and let
\[
S=\{2^n+Q : n \in \mathbb{N}\cup\{0\}\}.
\]
Then $S$ is a multiplicative set but not a strongly multiplicative one, because
\[
\bigcap_{n \geq 0}(2^n+Q)R=0.
\]
Set
\[
\mathcal P=(3,X_1,X_2,\ldots)/Q.
\]
As in \cite{Hamed}, $\mathcal P$ is an $S$-prime ideal, and so is
\[
\mathcal P_n=(2^n+Q)\mathcal P
\qquad\text{for each } n \geq 0.
\]
The ideals $\mathcal P_n$ form a descending chain, but $\bigcap_{n \geq 0}\mathcal P_n=0$ is not $S$-prime.
Indeed, for a fixed $2^n+Q \in S$,
\[
(2^{n+2}+Q)(X_{n+2}+Q)=0,
\]
while neither $(2^n+Q)(2^{n+2}+Q)$ nor $(2^n+Q)(X_{n+2}+Q)$ is zero in $R$.
Thus the intersection of a descending chain of $S$-prime ideals need not be $S$-prime when $S$ is not strongly multiplicative.
\end{example}

\begin{example}[Non-necessity in a special case]\label{ex:counter2}
Let $R=\mathbb{Z}$ and let $S=\mathbb{Z}\setminus 2\mathbb{Z}$.
Then $S$ is not strongly multiplicative, because $2\mathbb{Z}$ is not strongly prime and Theorem~\ref{thm:strongprime} applies.
For each $n \geq 0$, let $I_n=3^n2\mathbb{Z}.$
Since $(I_n:3^n)=2\mathbb{Z}$ is prime, each $I_n$ is $S$-prime by \cite[Proposition~1]{Hamed}.
The ideals $I_n$ form a descending chain, and $\bigcap_{n \geq 0} I_n=0$ is again $S$-prime.
So the strong multiplicativity of $S$ is not necessary in every isolated example.
\end{example}

The next examples address minimal $S$-primes.

\begin{example}[A ring without $S$-minimal primes]\label{ex:counter3}
Let $R$ and $S$ be as in Example~\ref{ex:counter1}.
We claim that $R$ has no $S$-minimal prime ideal.
Suppose that $\mathcal P$ is such an ideal.
By \cite[Proposition~2]{Hamed}, each $(2^n+Q)\mathcal P$ is $S$-prime.
Minimality of $\mathcal P$ forces
\[
\mathcal P=(2^n+Q)\mathcal P \subseteq (2^n+Q)R
\qquad\text{for every } n,
\]
so
\[
\mathcal P \subseteq \bigcap_{n \geq 0}(2^n+Q)R=0.
\]
Hence $\mathcal P=0$, but Example~\ref{ex:counter1} shows that $0$ is not $S$-prime.
This contradiction proves the claim.
\end{example}

\begin{example}[A non-strong example with a unique $S$-minimal prime]\label{ex:counter4}
Let $R=\mathbb{Z}\times\mathbb{Z}$ and let
\[
S=\Reg(\mathbb{Z})\times\{1\}=(\mathbb{Z}\setminus\{0\})\times\{1\}.
\]
Then $S$ is multiplicative but not strongly multiplicative, because
\[
\bigcap_{0 \neq x \in \mathbb{Z}}(x,1)R=(0,1)R
\]
is disjoint from $S$.
The zero ideal is not $S$-prime.
A direct inspection of ideals disjoint from $S$ shows that $\{0\}\times\mathbb{Z}$ is the unique $S$-minimal prime ideal of $R$.
Thus the existence of an $S$-minimal prime does not force $S$ to be strongly multiplicative.
\end{example}

The next theorem isolates the consequences that do hold under strong multiplicativity.

\begin{theorem}\label{thm:sminimal}
Let $S$ be a multiplicative set of $R$.
\begin{enumerate}
\item If $P$ is a minimal $S$-prime ideal of $R$, then $sP=P$ for every $s \in S$.
In particular, $P \subseteq \bigcap_{s \in S} sR.$
\item Suppose that $S$ is strongly multiplicative.
If $S \nsubseteq U(R)$, then no $S$-minimal prime ideal is a prime ideal.
If $S \subseteq U(R)$, then every $S$-prime ideal is an ordinary prime ideal, hence every $S$-minimal prime is a minimal prime of $R$.
\item If $S$ is strongly multiplicative and $S \nsubseteq U(R)$, then $R$ possesses $S$-minimal prime ideals, and all of them are non-prime.
\item Let $I$ be an ideal with $I \cap S=\varnothing$, and assume that $S$ is strongly multiplicative.
If $I+\bigcap_{s \in S} sR \neq R,$ then every $S$-prime ideal minimal over $I$ is non-prime.
If $I+\bigcap_{s \in S} sR=R,$ then every $S$-prime ideal minimal over $I$ is an ordinary minimal prime over $I$.
\end{enumerate}
\end{theorem}

\begin{proof}
(1) Let $P$ be a minimal $S$-prime ideal.
By \cite[Proposition~2]{Hamed}, $sP$ is $S$-prime for every $s \in S$.
Since $sP \subseteq P$, minimality forces $sP=P$.
Consequently $P \subseteq sR$ for every $s \in S$, and therefore
$P \subseteq \bigcap_{s \in S} sR.$

(2) Assume that $S$ is strongly multiplicative and $S \nsubseteq U(R)$.
Let $P$ be an $S$-minimal prime ideal, and suppose that $P$ is prime.
By (1), $P \subseteq \bigcap_{s \in S} sR.$
Choose a least element $t \in S$ and write $t=ue$ as in Lemma~\ref{lem:idempotent}.
Then $P \subseteq tR=eR.$
Since $S \nsubseteq U(R)$, the idempotent $e$ is nontrivial.
Because $e(1-e)=0 \in P$ and $P$ is prime, either $e \in P$ or $1-e \in P$.
The second option is impossible because $1-e \in P \subseteq eR$ would imply $e=1$.
Hence $e \in P$, and therefore $t=ue \in P \cap S,$ a contradiction.
Thus no $S$-minimal prime is prime when $S \nsubseteq U(R)$.

If $S \subseteq U(R)$, then the definition of $S$-prime reduces to the usual definition of prime ideal, because the distinguished element $s \in S$ is invertible.

(3) The existence of minimal $S$-primes under strong multiplicativity is proved in \cite[Proposition~5]{Hamed}.
The non-primality follows from (2).

(4) Assume first that $I+\bigcap_{s \in S} sR=R.$
Choose $x \in \bigcap_{s \in S} sR$ with $1-x \in I$.
Passing to the quotient ring $R/I$, we obtain
\[
1+I=x+I \in \bigcap_{\overline s \in \overline S}\overline s(R/I),
\]
where $\overline S$ denotes the image of $S$ in $R/I$.
Hence every element of $\overline S$ divides $1+I$, so $\overline S \subseteq U(R/I)$.
Therefore every $\overline S$-prime ideal of $R/I$ is an ordinary prime ideal.
By \cite[Proposition~3]{Hamed}, $\overline S$-prime ideals of $R/I$ correspond exactly to $S$-prime ideals of $R$ containing $I$.
Thus every $S$-prime ideal minimal over $I$ is an ordinary minimal prime over $I$.

Assume now that $I+\bigcap_{s \in S} sR \neq R.$
By Corollary~\ref{cor:factor}, the image $\overline S$ of $S$ in $R/I$ is strongly multiplicative.
Moreover,
\[
\bigcap_{\overline s \in \overline S}\overline s(R/I)
=
\frac{I+\bigcap_{s \in S}sR}{I}
\neq R/I,
\]
so $\overline S \nsubseteq U(R/I)$.
Applying (3) in the quotient ring shows that every $\overline S$-minimal prime of $R/I$ is non-prime.
By \cite[Proposition~3]{Hamed}, these correspond exactly to the $S$-prime ideals of $R$ minimal over $I$.
Hence every $S$-prime ideal minimal over $I$ is non-prime.
\end{proof}

\begin{corollary}
Let $R=R_1\times R_2$, where $R_1$ and $R_2$ are not fields, and let $S$ be a strongly multiplicative set of $R$ such that $S\not\subseteq U(R)$. Then $R$ has $S$-minimal prime ideals, and none of them is a prime ideal of $R$. Moreover, if $\overline{S}=U(R_1)\times R_2$, then every $S$-minimal prime ideal of $R$ is of the form $P\times \{0\}$, where $P$ is a minimal prime ideal of $R_1$. Dually, if $\overline{S}=R_1\times U(R_2)$, then every $S$-minimal prime ideal of $R$ is of the form $\{0\}\times P$, where $P$ is a minimal prime ideal of $R_2$.
\end{corollary}

\section{Strong Krull separation and maximal correspondence}

Classical Krull separation guarantees prime ideals disjoint from a multiplicative set, but not maximal ideals.
For instance, if $R=k[X,Y]$ and $S=R\setminus (X)$, where $k$ is a field and $X,Y$ are indeterminates over $k$, then $(X)$ is maximal among ideals disjoint from $S$, yet it is not a maximal ideal of $R$. Here, note that $S$ is not strongly multiplicative since $(X)$ is not a strongly prime ideal of $R$. For strongly multiplicative sets the situation improves drastically.

\begin{theorem}[Strong Krull separation]\label{thm:strongkrull}
Let $S$ be a strongly multiplicative set of $R$, and let $I$ be an ideal with $I \cap S=\varnothing$.
Then there exists a maximal ideal of $R$ containing $I$ and disjoint from $S$.
Moreover, every ideal maximal with respect to containing $I$ and being disjoint from $S$ is a maximal ideal of $R$.
\end{theorem}

\begin{proof}
Choose an idempotent $e$ such that $R_S \cong Re$.
Since $I \cap S=\varnothing$, the ideal $Ie$ is proper in $Re$; otherwise $e \in I$, and then a least element $t=ue$ of $S$ would lie in $I \cap S$.
Choose a maximal ideal $N$ of $Re$ containing $Ie$, and let $M$ be its inverse image in $R$.
Then $M$ is a maximal ideal of $R$, it contains $I$, and $M \cap S=\varnothing$ because every $s \in S$ maps to a unit in $Re$.
This proves existence.

Now let $P$ be maximal with respect to containing $I$ and being disjoint from $S$.
By the classical Krull separation lemma, $P$ is prime.
Let $t=ue$ be a least element of $S$.
Since $t \notin P$, also $e \notin P$.
From $e(1-e)=0 \in P$ and primality of $P$, we get $1-e \in P$.

Take any $x \notin P$.
By maximality of $P$, the ideal $P+Rx$ meets $S$, so choose $s \in S \cap (P+Rx)$.
Modulo $P$ one has $\overline s=\overline r\,\overline x$ for some $\overline r \in R/P$.
Because every element of $S$ divides $t$ and $t$ maps to $\overline u$, a unit of $R/P$, the class $\overline s$ is a unit in $R/P$.
Hence $\overline x$ is a unit as well.
Therefore every nonzero class in $R/P$ is invertible, so $R/P$ is a field and $P$ is maximal.
\end{proof}

Disjoint maximal ideals need not be strongly prime, even in the strongly multiplicative setting.

\begin{example}\label{ex:max-not-strong}
Let $A=C[0,1]$ be the ring of real-valued continuous functions on $[0,1]$, and let
\[
M_r=\{f \in C[0,1] : f(r)=0\}
\]
for $r \in [0,1]$.
Each maximal ideal of $A$ is of this form \cite{Atiyah}.
Set $R=A \times \mathbb{R}$ and
\[
S=\{(1,0),(1,1)\}.
\]
Then $S$ is strongly multiplicative.
Any maximal ideal of $R$ disjoint from $S$ must be of the form $M_r \times \mathbb{R}$.
However, such an ideal is not strongly prime:
if $M_i^\ast=M_i \times \mathbb{R}$, then
\[
\bigcap_{i \neq r} M_i^\ast=0 \times \mathbb{R}\subseteq M_r^\ast,
\]
while no $M_i^\ast$ with $i \neq r$ is contained in $M_r^\ast$.
\end{example}

\begin{corollary}\label{cor:maxunion}
If $S$ is saturated and strongly multiplicative, then $R \setminus S$ is a union of maximal ideals of $R$.
\end{corollary}

\begin{proof}
For every saturated multiplicative set,
$R \setminus S=\bigcup_{\mathfrak p \cap S=\varnothing}\mathfrak p.$
By Theorem~\ref{thm:strongkrull}, each prime ideal disjoint from $S$ is contained in a maximal ideal disjoint from $S$.
Hence
$R \setminus S=\bigcup_{M \cap S=\varnothing} M.$
\end{proof}

The idempotent description of $R_S$ also yields a clean description of maximal ideals after localization.

\begin{theorem}[Maximal correspondence]\label{thm:maxcorr}
Let $S$ be a strongly multiplicative set of $R$.
Then extension and contraction induce a bijection between
\[
\Max(R_S)
\qquad\text{and}\qquad
\{M \in \Max(R) : M \cap S=\varnothing\}.
\]
\end{theorem}

\begin{proof}
Choose an idempotent $e$ with $R_S \cong Re$.
The map $R \to Re$, $r \mapsto re$, is surjective with kernel $(1-e)R$, so maximal ideals of $Re$ correspond to maximal ideals of $R$ that contain $1-e$.
We claim that these are exactly the maximal ideals of $R$ disjoint from $S$.

Let $M$ be maximal and contain $1-e$.
If $M$ met $S$, pick $s \in M \cap S$.
A least element $t=ue$ of $S$ is divisible by $s$, so $t \in M$, whence $e=u^{-1}t \in M$.
Together with $1-e \in M$, this gives $1 \in M$, a contradiction.
Thus $M \cap S=\varnothing$.

Conversely, let $M$ be maximal with $M \cap S=\varnothing$.
Then $t=ue \notin M$, so $e \notin M$.
Since $M$ is prime and $e(1-e)=0$, we must have $1-e \in M$.
Therefore $M$ contains $(1-e)R$.

Thus maximal ideals of $Re$ are exactly the maximal ideals of $R$ disjoint from $S$.
Because $R_S \cong Re$, the result follows.
\end{proof}

\section{The regular \texorpdfstring{$m$}{m}-complement operator}

We now restrict to regular elements.
Write
\[
\Reg(R)=\{r \in R : rx=0 \Rightarrow x=0\}.
\]
Thus $\Reg(R)$ is the set of non-zerodivisors of $R$.
For a nonempty subset $S \subseteq \Reg(R)$ define
\[
N_R(S)=\{x \in \Reg(R) : xR \cap sR=xsR \text{ for all } s \in S\}.
\]
When $R$ is an integral domain, this is the Anderson--Chang $m$-complement \cite{AndersonChang}.

\begin{proposition}\label{prop:mcomp}
Let $R$ be a commutative ring.
For nonempty subsets $S,S_1,S_2 \subseteq \Reg(R)$ and any family $(S_\alpha)_\alpha$ of nonempty subsets of $\Reg(R)$, the following hold.
\begin{enumerate}
\item $N_R(S)$ is a saturated multiplicative subset of $\Reg(R)$.
\item If $S_1 \subseteq S_2$, then $N_R(S_2)\subseteq N_R(S_1)$.
\item $S \cap N_R(S)\subseteq U(R)$.
\item $S \subseteq N_R(N_R(S))$.
\item $N_R(N_R(N_R(S)))=N_R(S)$.
\item
\[
N_R\Bigl(\bigcup_\alpha S_\alpha\Bigr)=\bigcap_\alpha N_R(S_\alpha).
\]
\item
\[
N_R(S_1S_2)=N_R(S_1)\cap N_R(S_2),
\]
where $S_1S_2=\{s_1s_2 : s_i \in S_i\}$.
\end{enumerate}
\end{proposition}

\begin{proof}
We first prove (1).
Let $x,y \in N_R(S)$ and fix $s \in S$.
The inclusion $xysR \subseteq xyR \cap sR$ is clear.
Conversely, if $w \in xyR \cap sR$, then $w \in xR \cap sR=xsR$, say $w=xsa$.
Also $w=xyb$ for some $b$.
Since $x$ is regular, cancellation gives $sa=yb$.
Now $sa \in yR \cap sR=ysR$, so $sa=ysc$ for some $c$.
Hence
\[
w=xsa=xysc \in xysR.
\]
Thus $xyR \cap sR=xysR$, and $xy \in N_R(S)$.

To show saturation, suppose $xy \in N_R(S)$.
Fix $s \in S$ and let $w \in xR \cap sR$.
Then
\[
yw \in xyR \cap sR=xysR,
\]
so $yw=xysc$ for some $c$.
Since $y$ is regular, cancellation gives $w=xsc$, and therefore $xR \cap sR=xsR$.
Hence $x \in N_R(S)$, and similarly $y \in N_R(S)$.

Statement (2) is immediate from the definition, and (6) follows by unpacking the universal quantifier.

For (3), let $s \in S \cap N_R(S)$.
Then
\[
sR=sR \cap sR=s^2R.
\]
Thus $s=s^2u$ for some $u \in R$.
Since $s$ is regular, cancellation yields $1=su$, so $s$ is a unit.

For (4), let $s \in S$ and $x \in N_R(S)$.
By definition,
\[
xR \cap sR=xsR=sxR,
\]
which says exactly that $s \in N_R(N_R(S))$.

For (5), apply (2) to the inclusion from (4) to obtain
\[
N_R(N_R(N_R(S)))\subseteq N_R(S).
\]
Applying (4) to $N_R(S)$ gives the reverse inclusion.

Finally, we prove (7).
If $x \in N_R(S_1)\cap N_R(S_2)$ and $s_i \in S_i$, then the multiplicativity argument already used above gives
\[
xR \cap s_1s_2R=xs_1s_2R,
\]
so $x \in N_R(S_1S_2)$.

Conversely, let $x \in N_R(S_1S_2)$, fix $s_1 \in S_1$, and choose $s_2 \in S_2$.
If $w \in xR \cap s_1R$, then
\[
s_2w \in xR \cap s_1s_2R=xs_1s_2R,
\]
so $s_2w=xs_1s_2c$ for some $c$.
Since $s_2$ is regular, cancellation gives $w=xs_1c$.
Thus $xR \cap s_1R=xs_1R$, and $x \in N_R(S_1)$.
By symmetry, $x \in N_R(S_2)$.
\end{proof}

\begin{theorem}\label{thm:mcomp}
Let $S \subseteq \Reg(R)$ be a nonempty multiplicative subset.
The following are equivalent.
\begin{enumerate}
\item $S$ is strongly multiplicative.
\item Every element of $S$ is a unit.
\item $S \subseteq N_R(S)$.
\item $N_R(S)=\Reg(R)$.
\end{enumerate}
If, in addition, $S$ is saturated, then these are also equivalent to
\begin{enumerate}
\setcounter{enumi}{4}
\item $S=U(R)$.
\end{enumerate}
When these conditions hold, $R_S \cong R$.
In particular, for every family of ideals $(I_\lambda)_{\lambda \in \Lambda}$,
\[
\Bigl(\bigcap_{\lambda \in \Lambda} I_\lambda\Bigr)_S
=
\bigcap_{\lambda \in \Lambda} (I_\lambda)_S
=
\bigcap_{\lambda \in \Lambda} I_\lambda.
\]
\end{theorem}

\begin{proof}
Assume that $S$ is strongly multiplicative, and let $t \in S$ be a least element with respect to inclusion of principal ideals.
Since $t^2 \in S$, we have $Rt\subseteq Rt^2$, so $t=t^2u$ for some $u \in R$.
Because $t$ is regular, cancellation gives $1=tu$, and therefore $t$ is a unit.
Now every $s \in S$ divides $t$, so every $s$ is a unit.
Thus (1) implies (2).

If every element of $S$ is a unit and $x \in \Reg(R)$, then for every $s \in S$,
\[
xR \cap sR=xR=xsR.
\]
Hence $x \in N_R(S)$, so $N_R(S)=\Reg(R)$.
Thus (2) implies (4), and (4) trivially implies (3).

If $S \subseteq N_R(S)$, then $S \subseteq S \cap N_R(S)$, so Proposition~\ref{prop:mcomp}(3) shows that every element of $S$ is a unit.
Hence (3) implies (2).

If $S$ is saturated and every element of $S$ is a unit, then $S=U(R)$.
Indeed, if $u \in S$ and $v \in U(R)$, then $u=(uv^{-1})v$, so saturation forces $v \in S$.
Thus (2) and (5) are equivalent.

Under any of these conditions, every element of $S$ is a unit.
Hence localization at $S$ is the identity.
\end{proof}

\begin{corollary}\label{cor:mcomp}
Let $S \subseteq \Reg(R)$ be a nonempty multiplicative subset.
The following are equivalent.
\begin{enumerate}
\item Localization at $S$ commutes with arbitrary intersections of ideals.
\item $S$ is strongly multiplicative.
\item Every element of $S$ is a unit.
\item $N_R(S)=\Reg(R)$.
\end{enumerate}
In particular, every nontrivial strongly multiplicative localization must invert a zero divisor.
\end{corollary}

\begin{proof}
The equivalence of (1) and (2) comes from Theorem~\ref{thm:main}, while Theorem~\ref{thm:mcomp} identifies (2), (3), and (4).
If localization at $S$ is nontrivial, then (3) fails, so some element of $S$ is a zero divisor.
\end{proof}

\begin{corollary}\label{cor:domain}
Let $D$ be an integral domain and let $S \subseteq D \setminus \{0\}$ be a nonempty multiplicative subset.
The following are equivalent.
\begin{enumerate}
\item $S$ is strongly multiplicative.
\item Every element of $S$ is a unit.
\item $S \subseteq N_D(S)$.
\item $N_D(S)=D \setminus \{0\}$.
\end{enumerate}
If, in addition, $S$ is saturated, then these are also equivalent to
\begin{enumerate}
\setcounter{enumi}{4}
\item $S=U(D)=N_D(N_D(S))$.
\end{enumerate}
\end{corollary}

\begin{proof}
Since $\Reg(D)=D \setminus \{0\}$, the first four assertions are exactly those of Theorem~\ref{thm:mcomp}.
If $S$ is saturated, then Theorem~\ref{thm:mcomp} gives $S=U(D)$.
For a unit $u$ and any nonzero $x \in D$,
\[
xD \cap uD=xD=xuD,
\]
so $N_D(U(D))=D \setminus \{0\}$.
Conversely, if $x \in N_D(D \setminus \{0\})$, then applying the defining equality to $s=x$ gives
\[
xD=xD \cap xD=x^2D,
\]
so $x$ is a unit.
Hence
\[
N_D(D \setminus \{0\})=U(D),
\]
and therefore
\[
N_D(N_D(S))=U(D)=S. \qedhere
\]
\end{proof}

\begin{corollary}\label{cor:tqr}
For a ring $R$, the following are equivalent.
\begin{enumerate}
\item $\Reg(R)$ is strongly multiplicative.
\item $\Reg(R)\subseteq U(R)$.
\item $R$ is a total quotient ring.
\end{enumerate}
\end{corollary}

\begin{proof}
Apply Theorem~\ref{thm:mcomp} to the multiplicative set $\Reg(R)$.
The condition $\Reg(R)\subseteq U(R)$ is exactly the definition of a total quotient ring.
\end{proof}

\bigskip
\noindent\textbf{Declarations}

\medskip
\noindent\textbf{Conflict of interest.} The authors declare that they have no conflict of interest.

\medskip
\noindent\textbf{Data availability.} No datasets were generated or analyzed in the preparation of this article.

\end{document}